\documentclass[reqno,b5paper]{amsart}
\usepackage{amsmath}
\usepackage{amssymb}
\usepackage{amsthm}
\usepackage{enumerate}
\usepackage[mathscr]{eucal}
\theoremstyle{plain}
\newtheorem{thm}{Theorem}[section]

\newtheorem{note}{Note}[section]

\theoremstyle{definition}
\newtheorem{defn}{Definition}[section]

\begin{document}

\setcounter {page}{1}
\title{Rough $I$-statistical convergence of sequences}

\author[P. Malik, M. Maity and AR. Ghosh]{ Prasanta Malik*, Manojit Maity** and Argha Ghosh*\ }
\newcommand{\acr}{\newline\indent}
\maketitle
\address{{*\,} Department of Mathematics, University of Burdwan, Golapbag, Burdwan-713104,
West Bengal, India.
                Email: pmjupm@yahoo.co.in. , buagbu@yahoo.co.in \acr
           {**\,} 25,Teachers Housing Estate, P.O.- Panchasayar, Kolkata-700094, India. Email:
           manojitm@research.jdvu.ac.in\\}

\maketitle
\begin{abstract}
The concept of rough convergence of sequences was first introduced by H.X. Phu \cite{Ph2}. Gradually a lot of work had been done on this concept and a new version of statistical convergence had come through the work of S. Aytar \cite{Ay2}. Through further progress of this concept, $I$-statistical convergence was then introduced by U. Yamanc and M. G$\bar{u}$rdal \cite{Ya}. The concept of rough ideal statistical convergence of a sequence was first defined by Das, Savas and Ghosal \cite{Da4}.
Here in this paper using the concept of Das et.al.  we prove some results on rough ideal statistical convergence and also we introduce the notion of rough ideal limit set and discuss some topological aspects on this set.

\end{abstract}
\author{}
\maketitle
{ Key words and phrases : Ideal statistical convergent, Rough ideal statistical convergent,
$I$-statistical limit set, \textit{I}- statistical bounded.}

\textbf {AMS subject classification (2010) : 40A05, 40G99} .  \\

\baselineskip 1cm

\section{\textbf{Introduction and background}}

The notion of convergence of real sequences has been extended to statistical convergence by Fast \cite{Fa} also independently by Schoenberg \cite{Sc} using the natural density of $\mathbb N$. The
study of statistical convergence become one of the most active research area in summability theory after the works of Fridy \cite{Fr1}, Salat \cite{Sal} and many others.

 The notion of statistical convergence has been further generalized to $I$-convergence by Kostyrko et.al. \cite{Ko1} using ideals of $\mathbb N$. A lot of work on $I$-convergence can be found in \cite{Da1, Da3} and many others. Recently Das et.al. in \cite{Da3} introduced the notion of ideal statistical convergence which is a new generalization of the notion of statistical convergence. More investigation and application of this notion can be found in (see \cite{Da4}, \cite{Da5}, \cite{Sav}, \cite{Ya}).

The concept of rough  convergence was first introduced by Phu \cite{Ph2}.
Further this notion of rough convergence has been extended to rough statistical convergence by Aytar \cite{Ay2} using the notion of natural density of $\mathbb N$ in a similar way as usual notion of convergence was extended to statistical convergence. More work can be found in (\cite{Ma1}, \cite{Ma2}). Further the notion of rough statistical convergence was generalized to rough $I$-convergence by Pal et.al. \cite{Pal} using ideals of $\mathbb N$. More investigation and application on this line can be found in (\cite{Ma3}, \cite{Pal}). So naturally one can think if the new notion of $I$-statistical convergence can be introduced in the theory of rough convergence.

 In this chapter we introduce and study the notion of rough $I$-statistical convergence in a normed linear space $\left(X,\left\|.\right\|\right)$ which naturally extends the notions of rough convergence as well as rough statistical convergence in a new way. We also define the set of all rough $I$-statistical limits of a sequence and proved some topological properties of this set.

\section{\textbf{Basic Definitions and Notations}}
In this section we recall some basic definitions and notations.
\begin{defn}
 Let $X\neq\phi $. A class $ I $ of subsets of $X$ is said to be
an ideal in X provided, $I$ satisfies the conditions:
\\(i)$\phi \in I$,
\\(ii)$ A,B \in I \Rightarrow A \cup B\in I,$
\\(iii)$ A \in I, B\subset A \Rightarrow B\in I$.
\end{defn}

An ideal $I$ in a non-empty set $X$ is called non-trivial if $ X \notin I.$
\begin{defn}
 Let $X\neq\phi  $. A non-empty class $\mathbb F $ of subsets of $X$ is
said to be a filter in $X$ provided that:
\\(i)$\phi\notin \mathbb F $,
\\(ii) $A,B\in\mathbb F \Rightarrow A \cap B\in\mathbb F,$
\\(iii)$ A \in\mathbb F, B\supset A \Rightarrow B\in\mathbb F$.
\end{defn}
\begin{defn}
 Let $I$ be a non-trivial ideal in a non-empty set $ X$.
 Then the class $\mathbb F(I)$$ = \left\{M\subset X : \exists A \in I ~such~~ that ~M = X\setminus A\right\}$
is a filter on $X$. This filter $\mathbb F(I) $ is called the filter associated with $I$.
\end{defn}
 A non-trivial ideal $I$ in $X(\neq \phi)$ is called admissible if $\left\{x\right\} \in I$ for each $x \in X$.
Throughout the paper we take $I$ as a non-trivial admissible ideal in $\mathbb{N} $ unless otherwise mentioned.
\begin{defn}\cite{Da4}
Let $x=\{x_k\}_{k \in \mathbb{N}}$ be a sequence of real numbers.
Then $x$ is said to be
 $I$-statistically convergent to $ \xi $
if for any $\varepsilon > 0$ and $ \delta > 0$
\begin{center}
$\{ n \in \mathbb{N}: \frac{1}{n}|\{ k \leq n: \parallel x_k - \xi
\parallel \geq  \varepsilon\}| \geq \delta \} \in I $.
\end{center}
In this case we write $ {I\mbox{-}st\mbox{-}\lim}~ x = \xi $
\end{defn}
\begin{defn}\cite{Ko1}
Let $x=\{x_n\}_{n \in \mathbb{N}}$ be a sequence of real numbers.
Then $x$ is said to be $I$-convergent to $ \xi $
if for any $\varepsilon > 0$
\begin{center}
$\{ n : \parallel x_n - \xi
\parallel \geq  \varepsilon\} \in I $.
\end{center}
In this case we write $ I-\underset{n\rightarrow \infty}{\lim}~ x_n = \xi $
\end{defn}
\begin{defn}\cite{Da5}
A subset $K$ of $\mathbb N$ is said to have $I$-natural density
$d_I(K)$ if
 \begin{center}
 $d_I(K)=I-\underset{n\rightarrow \infty}{\lim}\frac{\left|K(n)\right|}{n}$
 \end{center}
 where $K(n)=\left\{j\in K:j\leq n\right\}$ and $\left|K(n)\right|$ represents the number of elements in $K(n)$.
\end{defn}
\begin{note}
From the above definition, it is clear that, if $d(A) = r, A \subset \mathbb{N}$, then $d_I (A) = r $
for any nontrivial admissible ideal $I$ in $\mathbb{N}$.
\end{note}
\begin{defn}\cite{Ph2}
If $ x = \{ x_{k}\}_{ k \in \mathbb{N}}$ be a sequence in some
normed linear space $ ( X, \| . \| )$ and $r$ be a non negative
real number then $x$ is said to be $r$-convergent to $ \xi \in X
$, if for any $ \varepsilon > 0 $, there exists $ N \in \mathbb{N}$
such that $\parallel {x_{k}-\xi}\parallel   < r + {\varepsilon}
~~~~~~for ~ all ~ k \geq N$.
\end{defn}
\begin{defn}\cite{Ay2}
Let $x = \{x_{k}\}_{k\in \mathbb{N}}$ be a sequence in a normed
linear space $(X, \parallel.\parallel)$ and $r$ be a non negative
real number. $x$ is said to be $r$- statistically convergent to
$\xi$, denoted by $x \overset{r-st}\longrightarrow \xi$, if for
any $\varepsilon > 0$ we have $d(A(\varepsilon)) = 0$, where
$A(\varepsilon) = \{ k \in \mathbb{N} : \parallel x_{k} - \xi
\parallel \geq r + \varepsilon \}$. In this case $\xi$ is called
the $r$-statistical limit of $x$.
\end{defn}

\begin{defn}
Let $x=\{x_k\}_{k \in \mathbb{N}}$ be a sequence in  a normed
linear space $\left(X,\left\|.\right\|\right)$ and $r$ be a non
negative real number. Then $x$ is said to be rough
$I$-statistically convergent to $ \xi $ or
$r\mbox{-}I$-statistically convergent to $\xi$ if for any
$\varepsilon > 0$ and $ \delta > 0$
\begin{center}
$\{ n \in \mathbb{N}: \frac{1}{n}|\{ k \leq n: \parallel x_k - \xi \parallel \geq r+ \varepsilon\}|
\geq \delta \} \in I $.
\end{center}
\end{defn}
In this case $\xi$ is called the rough $I$-statistical limit of
$x=\{x_k\}_{k \in \mathbb{N}}$ and we denote it by $ x_k
\overset{r\mbox{-}I \mbox{-}st}\longrightarrow \xi $.

Here $r$ in the above definition is called the roughness degree of
the rough $I$-statistical convergence. If $r=0$ we obtain the
notion of $I \mbox{-}$statistical convergence. But our main
interest is when $ r > 0 $. It may happen that a sequence $ x =
\{x_k\}_{k \in \mathbb{N}} $ is not $ I $-statistically convergent
in the usual sense, but there exists a sequence $ y = \{y_k\}_{k
\in \mathbb{N}} $, which is $ I $-statistically convergent and
satisfying the condition $\| x_k -y_k \|\leq r $ for all  $k$. Then $ x $ is
rough $I$-statistically convergent to the same limit.

From the above definition it is clear that the rough
$I$-statistical limit of a sequence is not unique. So we consider
the set of rough $I$-statistical limits of a sequence $x$ and we
use the notation $I\mbox{-}st\mbox{-}\mbox{LIM}_x^r$ to denote the
set of all rough $I$-statistical limits of a sequence $x$. We say
that a sequence $x$ is rough $I$-statistically convergent if
$I\mbox{-}st\mbox{-}\mbox{LIM}_x^r\neq \phi$.

Throughout the paper $x$ denotes the sequence $\{x_k\}_{k \in
\mathbb{N}}$ and $X$ denotes a normed linear space $(X, \|. \|)$.


\section{\textbf{Main Results}}

In this section we discuss some basic properties of rough $I$-statistical convergence of sequences.


\begin{thm}
Let $x = \{x_k \}_{ k \in \mathbb{N}} $ be a sequence in $X$ and
$r>0 $. Then diam ($I\mbox{-}st\mbox{-}\mbox{LIM}_x^r) \leq 2r $.
In particular if $x$ is $I$-statistically convergent to $\xi$,
then $I\mbox{-}st\mbox{-}\mbox{LIM}_x^r
\supset{\overline{{B_r}(\xi)}}(=\left\{y\in
X:\left\|y-\xi\right\|\leq r\right\})$ and so diam
($I\mbox{-}st\mbox{-}\mbox{LIM}_x^r) = 2r $.
\end{thm}

\begin{proof}
If possible, let diam ($I\mbox{-}st\mbox{-}\mbox{LIM}_x^r) > 2r $.
Then there exist $ y,z \in I\mbox{-}st\mbox{-}\mbox{LIM}_x^r $
such that $\| y - z \| > 2r $. Now choose $ \varepsilon > 0$ so
that $ \varepsilon < \frac{\| y - z \|}{2} - r $. Let $ A = \{k
\in \mathbb{N} : \|x_k - y \| \geq r +\varepsilon  \}$ and $ B =
\{k \in \mathbb{N} : \|x_k - z \| \geq r + \varepsilon \}$. Then
\begin{center}
$ \frac{1}{n} |\{ k \leq n : k \in A \cup B \}| \leq
\frac{1}{n}|\{ k \leq n : k \in A\}| + \frac{1}{n}|\{ k \leq n : k
\in B \}|$
\end{center}
and so by the property of $I$-convergence we have,
\begin{center}
 $I \mbox{-}\underset{n \rightarrow \infty}{\lim}
\frac{1}{n}|\left\{k\leq n:k\in A\cup B\right\}| \leq I
\mbox{-}\underset{n \rightarrow \infty}{\lim}\frac{1}{n}|\{k \leq
n: k \in A \}| + I \mbox{-}\underset{n \rightarrow
\infty}{\lim}\frac{1}{n}|\{k \leq n: k \in B \}| = 0 $.
\end{center}
Thus
\begin{center}
$\{n \in \mathbb{N} : \frac{1}{n} |\{ k \leq n : k \in A \cup B \}| \geq \delta\} \in I $ for all $\delta>0$.
\end{center}
 Let
\begin{center}
$ K = \{ n \in \mathbb{N} : \frac{1}{n} |\{k \leq n : k \in A \cup B \}| \geq \frac{1}{2}\} $.
\end{center}
Clearly $K\in I$. Now choose $ n_0 \in \mathbb{N} \setminus K $. Then
\begin{center}
$ \frac{1}{n_0}|\{k \leq n_0 : k \in A \cup B \}| < \frac{1}{2}$.
\end{center}
 So
 \begin{center}
 $\frac{1}{n_0} | \{k \leq n_0: k \notin A \cup B \}| \geq 1- \frac{1}{2} = \frac{1}{2} $
\end{center}
 i.e., $ \{ k : k \notin A \cup B \}$ is a nonempty set.\\
Take $ k_0 \in \mathbb{N}$
such that $ k_0 \notin A \cup B$. Then $ k_0 \in A^c \cap B^c$ and
hence $ \|x_{k_0} - y \| < r + \varepsilon $ and $ \| x_{k_0} - z
\| < r + \varepsilon $. So
\begin{center}
$\|y - z \| \leq \|x_{k_0} -y \| + \|x_{k_0} - z \| \leq 2(r + \varepsilon) < \| y - z \|$,
\end{center}
 which is
absurd. Therefore diam ($I\mbox{-}st\mbox{-}\mbox{LIM}_x^r) \leq
2r $.

Now if $ I\mbox{-}st\mbox{-}\lim x = \xi $, we proceed as follows.
Let $ \varepsilon
> 0 $ and $ \delta > 0 $ be given. Then
\begin{center}
$ A = \{n \in \mathbb{N}: \frac{1}{n}|\{ k \leq n : \|x_k - \xi \|\geq \varepsilon \}|\geq \delta \} \in I $.
\end{center}
Then for $ n \notin A $ we have
\begin{center}
$\frac{1}{n}|\{k \leq n : \|x_k - \xi \| \geq \varepsilon \} < \delta $,
\end{center}
 i.e.,
\begin{eqnarray}
\frac{1}{n} |\{k \leq n : \| x_k - \xi \| < \varepsilon \} \geq 1 - \delta.
\end{eqnarray}
Now for each $y\in {\overline{{B_r}(\xi)}}=\left\{y\in X:\left\|y-\xi\right\|\leq r\right\}$ we have
\begin{eqnarray}
\|x_k - y \| \leq \| x_k - \xi \| + \|\xi - y \| \leq \|x_k - \xi \| + r.
\end{eqnarray}
Let $B_n = \left\{ k \leq n : \|x_k - \xi \| < \varepsilon \right\}$.
Then for $ k \in B_n$ we have $\|x_k -y \| < r + \varepsilon$.
Hence
\begin{center}
$ B_n \subset \{ k \leq n: \|x_k - y \|< r + \varepsilon\}$.
\end{center}
This implies,
\begin{center}
$ \frac{|B_n|}{n} \leq
\frac{1}{n}|\{k \leq n: \|x_k -y \| < r + \varepsilon\}|$
\end{center}
 i.e.,
\begin{center}
$\frac{1}{n}|\{k \leq n:\|x_k - y\|< r + \varepsilon \}| \geq 1 -\delta $.
\end{center}
Thus for all $ n \notin A $,
\begin{center}
$ \frac{1}{n} |k \leq n: \|x_k - y \|\geq r + \varepsilon \}|<1-(1 - \delta) $.
\end{center}
Hence we have
\begin{center}
$\{n \in \mathbb{N}: \frac{1}{n} |\{k \leq n:\|x_k - y \| \geq r +\varepsilon \}| \geq \delta\} \subset A $.
\end{center}
Since $ A \in I $, so
\begin{center}
$\{n \in \mathbb{N}: \frac{1}{n} |\{k \leq n:\|x_k - y \| \geq r + \varepsilon \}| \geq \delta\} \in I $.
\end{center}
This shows that $ y \in I\mbox{-}st\mbox{-}LIM_x^r$. Therefore
$I\mbox{-}st\mbox{-}\mbox{LIM}_x^r \supset{\overline{{B_r}(\xi)}}$
and consequently diam$(I\mbox{-}st\mbox{-}LIM_x^r) \geq 2r$. Hence
diam$(I\mbox{-}st\mbox{-}LIM_x^r) = 2r$.
\end{proof}

In the paper \cite{Ph2} H.X. Phu has already shown that for any
subsequence $ {x'} = \{x_{n_k}\}_{k \in \mathbb{N}}$ of a sequence
$x=\left\{x_k\right\}_{k\in\mathbb N}$, $ LIM_x^r \subseteq
LIM_{x'}^r $.

In the following theorem we show that the rough $I$-statistical analogue of Phu's result holds for some kind of subsequences.
\begin{thm}
Let $ x=\left\{x_k\right\}_{k\in\mathbb N} $ be a sequence in $X$
and $ r> 0 $ be any real number. If $ x $ has a subsequence $ x' =
\{ x_{j_k}\}_{k\in\mathbb N}$ satisfying the condition $ I\mbox{-}\underset{n
\rightarrow \infty}{\lim} \frac{1}{n} | \{ j_k \leq n : k \in
\mathbb{N}\}|=1 $ then $ I\mbox{-}st\mbox{-}LIM_x^r \subset
I\mbox{-}st\mbox{-}LIM_{x'}^r $.
\end{thm}

\begin{proof}
Let $ \xi \in I\mbox{-}st\mbox{-}LIM_x^r $. Let $\varepsilon>0$ be given. Since $ \xi \in I\mbox{-}st\mbox{-}LIM_x^r $ we have
\begin{center}
$I\mbox{-}\underset{n\rightarrow \infty}{\lim}\frac{1}{n}\left\{k\leq n:\left|x_k-\xi\right|\geq r+\varepsilon\right\}=0$
\end{center}
\begin{equation}
\Rightarrow I\mbox{-}\underset{n\rightarrow \infty}{\lim}\frac{1}{n}\left\{k\leq n:\left|x_k-\xi\right|<r+\varepsilon\right\}=1.~~~~~~
\end{equation}
Again by the given condition we have
\begin{equation}
 I\mbox{-}\underset{n\rightarrow \infty}{\lim} \frac{1}{n} | \{ j_k \leq n : k \in\mathbb{N}\}|=1.~~~~~~~~~~
\end{equation}
 Let $A=\left\{j_k:k\in\mathbb N\right\}$. Then by (4.4) we have $d_I(A)=1$ and so $d_I( \mathbb{N} \setminus A) = 0 $ . Then from (4.3) we have
\begin{equation}
I\mbox{-}\underset{n\rightarrow \infty}{\lim}\frac{1}{n}\left\{j_k\leq n:\left|x_{j_k}-\xi\right|< r+\varepsilon\right\}=1.~~~~~~
\end{equation}
Now
\begin{equation}
1\geq \frac{1}{n}\left|\left\{k\leq n:\left|x_{j_k}-\xi\right|<r+\varepsilon\right\}\right|=\frac{1}{n}\left|\left\{j_k\leq j_n:\left|x_{j_k}-\xi\right|<r+\varepsilon\right\}\right|.~~~~~~~~~~~~~~~~~~~~~~
\end{equation}
Also
\begin{equation}
 \frac{1}{n}\left|\left\{j_k\leq j_n:\left|x_{j_k}-\xi\right|<r+\varepsilon\right\}\right|\geq \frac{1}{n}\left|\left\{j_k\leq n:\left|x_{j_k}-\xi\right|<r+\varepsilon\right\}\right|.~~~~~~~~~
\end{equation}
Thus from (4.5), (4.6), (4.7) and by the property of $I$-convergence we have
\begin{center}
$I\mbox{-}\underset{n\rightarrow \infty}{\lim}\frac{1}{n}\left\{k\leq n:\left|x_{j_k}-\xi\right|< r+\varepsilon\right\}=1$
\end{center}
\begin{center}
$\Rightarrow I\mbox{-}\underset{n\rightarrow \infty}{\lim}\frac{1}{n}\left\{k\leq n:\left|x_{j_k}-\xi\right|\geq r+\varepsilon\right\}=0.$
\end{center}
Hence $x'$ is rough $I$-statistical convergent to $\xi$ i.e., $\xi\in I\mbox{-}st\mbox{-}LIM_{x'}^r$. Since $\xi$ is arbitrary, we have $ I\mbox{-}st\mbox{-}LIM_x^r \subset
I\mbox{-}st\mbox{-}LIM_{x'}^r $.
\end{proof}
\begin{thm}
Let $ x=\left\{x_k\right\}_{k\in\mathbb N} $ be sequence in $X$
and $ r > 0 $ be a real number. Then the rough $I$-statistical
limit set of the sequence $ x $ i.e., the set $
I\mbox{-}st\mbox{-}LIM_x^r $ is closed.
\end{thm}

\begin{proof}
If $ I\mbox{-}st\mbox{-}LIM_x^r  = \emptyset $, then nothing to prove.\\
Let us assume that $ I\mbox{-}st\mbox{-}LIM_x^r \neq \emptyset $. Now consider a sequence $\{y_k\}_{k \in \mathbb{N}}$
in $ I\mbox{-}st\mbox{-}LIM_x^r $ with $ \underset{k \rightarrow \infty}{\lim}y_k = y $.
Choose $ \varepsilon > 0 $ and $ \delta > 0 $. Then there exists $ i_{\frac{\varepsilon}{2}} \in \mathbb{N} $ such that $ \|y_k - y \| < \frac{\varepsilon}{2} $ for all $ k > i_{\frac{\varepsilon}{2}}$.
Let $ k_0 > i_{\frac{\varepsilon}{2}}$. Then $ y_{k_0} \in I\mbox{-}st\mbox{-}LIM_x^r $ and so
$ A = \{ n \in \mathbb{N} : \frac {1}{n} | \{ k \leq n: \|x_k - y_{k_0}\| \geq r + \frac{\varepsilon}{2}\}| \geq \delta \} \in I $.
Since $I$ is admissible so $ M = \mathbb{N} \setminus A $ is nonempty. Choose $ n \in M $.
Then
\begin{eqnarray*}
& &\frac{1}{n} | \{ k \leq n: \|x_k - y_{k_0} \| \geq r +
\frac{\varepsilon}{2} \} < \delta\\
&\Rightarrow & \frac{1}{n} | \{ k \leq n: \|x_k - y_{k_0} \| < r +
\frac{\varepsilon}{2} \} \geq 1 - \delta.
\end{eqnarray*}
Put $ B_n = \{ k \leq n: \|x_k - y_{k_0} \| < r +
\frac{\varepsilon}{2} \} $. Then for $ k \in B_n $,
\begin{eqnarray*}
\|x_k - y \| \leq \| x_k - y_{k_0}\| + \|y_{k_0} - y \| < r + \frac{\varepsilon}{2} + \frac{\varepsilon}{2}
= r + \varepsilon.
\end{eqnarray*}
Hence $ B_n \subset \{ k \leq n : \|x_k - y \| < r + \varepsilon
\}$, which implies $ 1-\delta \leq\frac{ |B_n|}{n} \leq
\frac{1}{n}\left|\{ k \leq n: \|x_k - y \| < r +
\varepsilon\}\right| $. Therefore $\frac{1}{n}|\{ k \leq n: \|x_k
- y \| \geq r + \varepsilon\}| < 1-(1-\delta)=\delta$. Then
\begin{eqnarray*}
\{ n: \frac{1}{n} | \{ k \leq n : \|x_k - y \| \geq r + \varepsilon \}| \geq \delta \} \subset A \in I.
\end{eqnarray*}
This shows that $ y \in I\mbox{-}st\mbox{-}LIM_x^r $. Hence $ I\mbox{-}st\mbox{-}LIM_x^r $ is a closed set.
\end{proof}
\begin{thm}
Let $ x =\left\{x_k\right\}_{k\in\mathbb N} $ be sequence in $X$
and $ r > 0 $ be a real number. Then the rough $I$-statistical
limit set $ I\mbox{-}st\mbox{-}LIM_x^r $ of the sequence $ x $ is
a convex set.
\end{thm}

\begin{proof}
Let $ y_0, y_1 \in I\mbox{-}st\mbox{-}LIM_x^r $ and $ \varepsilon
> 0 $ be given. Let
\begin{eqnarray*}
A_0 = \{ k \in \mathbb{N} : \| x_k - y_0 \| \geq r + \varepsilon \}~ \\
A_1 = \{ k \in \mathbb{N} : \| x_k - y_1 \| \geq r + \varepsilon \}.
\end{eqnarray*}
Then by theorem 4.3.1 for $ \delta > 0 $ we have $ \{ n :
\frac{1}{n} |\{k \leq n: k \in A_0 \cup A_1 \}| \geq \delta \} \in
I $. Choose $ 0 < {\delta}_1 < 1 $ such that $ 0 < 1 - {\delta}_1
< \delta $. Let $ A = \{ n : \frac{1}{n} |\{k \leq n: k \in A_0
\cup A_1 \}| \geq {1-\delta}_1 \}$. Then $ A \in I $. Now for each $
n \notin A $ we have
\begin{eqnarray*}
& &\frac{1}{n} |\{k \leq n: k \in A_0 \cup A_1 \}| < 1 -
{\delta}_1\\
& \Rightarrow & \frac{1}{n} |\{k \leq n: k \notin A_0 \cup A_1 \}|
\geq \left\{1-(1 - {\delta}_1)\right\}=\delta_1.
\end{eqnarray*}
Therefore $ \{ k \in \mathbb{N}: k \notin A_0 \cup A_1 \}$ is a
nonempty set. Let us take $ k_0 \in {A_0}^c \cap {A_1}^c $ and $ 0
\leq \lambda \leq 1 $. Then
\begin{eqnarray*}
\|x_{k_0} - (1 - \lambda)y_0 - {\lambda}y_1 \| & = & \|(1 -
\lambda)x_{k_0} + {\lambda}x_{k_0} - [(1 - \lambda)y_0 +
{\lambda}y_1] \| \\
& \leq & (1 - \lambda)\|x_{k_0} - y_0 \| + \lambda \|x_{k_0} -
y_1\| \\
& < & (1 - \lambda)(r + \varepsilon) + \lambda(r + \varepsilon) =
r + \varepsilon.
\end{eqnarray*}
Let $ B = \{ k \in \mathbb{N}: \|
x_k - [(1-\lambda)y_0 + {\lambda}y_1]\|\geq r + \varepsilon\}$.
Then clearly, $ {A_0}^c \cap {A_1}^c \subset B^c $. So for $ n
\notin A $,
\begin{eqnarray*}
& &{\delta}_1 \leq \frac{1}{n}|\{ k \leq n: k \notin A_0 \cup A_1
\} \leq \frac{1}{n}|\{ k \leq n: k \notin B \}\\
& \Rightarrow & \frac{1}{n}|\{ k \leq n : k \in B\}| < 1 -
{\delta}_1 < \delta.
\end{eqnarray*}
Thus $ A^c \subset \{ n : \frac{1}{n}|\{ k \leq n : k \in B\}| <
\delta \}$. Since $ A^c \in \mathbb F(I) $, so $\{ n :
\frac{1}{n}|\{ k \leq n : k \in B\}| < \delta\} \in\mathbb F(I)$
and so $\{ n : \frac{1}{n}|\{ k \leq n : k \in B\}| \geq \delta \}
\in I $. This completes the proof.
\end{proof}

\begin{thm}
Let $ r > 0 $. Then a sequence $ x=\left\{x_k\right\}_{k\in\mathbb
N} $ in $X$ is rough $I$-statistically convergent to $ \xi $ if
and only if there exists a sequence $ y = \{y_k\}_{k \in
\mathbb{N}} $ in $X$ such that $ {I\mbox{-}st\mbox{-}\lim}~ y =
\xi $ and $ \parallel x_{k} - y_{k}
\parallel \leq r $ for all $ k \in \mathbb{N} $.
\end{thm}

\begin{proof}
Let $ y = \{y_k\}_{k \in \mathbb{N}}$ be a sequence in $X$, which
is $I$- statistically convergent to $\xi$ and $\|x_k - y_k \| \leq
r $ for all $k\in\mathbb N$. Then for any $ \varepsilon > 0 $ and
$ \delta > 0 $ the set $ A = \{ n \in \mathbb{N}: \frac{1}{n}|\{ k
\leq n : \|y_k - \xi \|\geq \varepsilon \} \geq \delta \} \in I $.
Let $ n \notin A $. Then
\begin{eqnarray*}
& & \frac{1}{n} |\{k \leq n : \|y_k - \xi \| \geq \varepsilon \}|
< \delta\\
& \Rightarrow & \frac{1}{n}|\{ k \leq n : \|y_k - \xi \| <
\varepsilon \} \geq 1 - \delta.
\end{eqnarray*}
Let $ B_n = \{ k \leq n : \|y_k - \xi \| < \varepsilon \}$, $
n\in\mathbb N$. Then for $ k \in B_n $, we have
\begin{center}
$\|x_k - \xi \| \leq \| x_k - y_k \| + \|y_k - \xi \| < r +
\varepsilon.$
\end{center}
Therefore,
\begin{eqnarray*}
& ~& B_n \subset \{ k \leq n: \|x_k - \xi \| < r + \varepsilon
\} \\
& \Rightarrow & \frac{|B_n|}{n} \leq \frac{1}{n} | \{ k \leq n :
\|x_k - \xi \| < r + \varepsilon \}| \\
& \Rightarrow & \frac{1}{n} |\{ k \leq n: \|x_k - \xi \| < r +
\varepsilon \}| \geq 1 - \delta \\
& \Rightarrow & \frac{1}{n} |\{ k \leq n: \| x_k - \xi \| \geq r +
\varepsilon \}| < 1 - (1 - \delta) = \delta.
\end{eqnarray*}
Therefore, $\{ n \in \mathbb{N}: \frac{1}{n}|\{ k \leq n: \|x_k -
\xi \| \geq r + \varepsilon \} \geq \delta \} \subset A $ and
since $ A \in I $, we have $ \{ n \in \mathbb{N} : \frac{1}{n}|\{
k \leq n: \|x_k - \xi \| \geq r + \varepsilon \} \in I $. Hence $
x_k \overset{r\mbox{-}I\mbox{-}st}{\rightarrow} \xi $.

Conversely, suppose that $ x_k
\overset{r\mbox{-}I\mbox{-}st}{\longrightarrow} \xi $. Then for $
\varepsilon > 0 ~~\mbox{and}~~ \delta >0 $,
\begin{eqnarray*}
A = \{ n \in \mathbb{N} : \frac{1}{n} | \{ k \leq n : \|x_k - \xi \| \geq r + \varepsilon \}| \geq \delta \} \in I.
\end{eqnarray*}
Let $ n \notin A $. Then
\begin{eqnarray*}
\frac{1}{n} | \{ k \leq n : \|x_k - \xi \| \geq r + \varepsilon \}| < \delta.
\end{eqnarray*}
\begin{eqnarray*}
\Rightarrow\frac{1}{n} | \{ k \leq n : \|x_k - \xi \| < r + \varepsilon \}| \geq 1 - \delta .
\end{eqnarray*}
Let $ B_n= \{ k \leq n : \|x_k - \xi \| < r + \varepsilon \}$. Now
we define a sequence $y=\left\{y_k\right\}_{k\in \mathbb N}$ as
follows:
\[ y_k = \left\{
  \begin{array}{l l}
    \xi & \quad \text{,if $\|x_k - \xi \| \leq r $, }\\
    x_k + r \frac{\xi - x_k}{\|x_k - \xi \|}, & \quad \text{otherwise.}
  \end{array} \right.\]\\
Then
\begin{eqnarray*}
 \|y_k -x_k \| &=& \|\xi - x_k\| \leq r ~~~~,\mbox{if}~~~\|x_k -\xi \| \leq r ,\\
               &=& r ~~~~~~~~~~~~~~~~,~~~~otherwise.
\end{eqnarray*}
Also,
\begin{eqnarray*}
 \|y_k -\xi \| & = &  \left\{
  \begin{array}{l l}
    0, & \quad \text{if $\|x_k - \xi \| \leq r $ }\\
    \|x_k - \xi  + r \frac{\xi - x_k}{\|x_k - \xi \|}\|, & \quad \text{otherwise.}
  \end{array} \right.\\
  \\
& = &  \left\{
  \begin{array}{l l}
    0, & \quad \text{if $\|x_k - \xi \| \leq r $ }\\
    \| x_k - \xi \| - r , & \quad \text{otherwise.}
  \end{array} \right.\\
\end{eqnarray*}
Let $ k \in B_n $. Then
\begin{eqnarray*}
\|y_k - \xi \| & = &  0 , ~~\mbox{if}~~\|x_k -\xi \| \leq r,\\
& < & \varepsilon,   ~~\mbox{if}~~ r < \|x_k - \xi \| < r +
\varepsilon.
\end{eqnarray*}
Therefore
\begin{eqnarray*}
& & B_n \subset \{ k \leq n : \|y_k - \xi \| < \varepsilon \}\\
& \Rightarrow & \frac{|B_n|}{n} \leq \frac{1}{n} |\{ k \leq n:
\|y_k - \xi\| < \varepsilon\} | \\
& \Rightarrow & \frac{1}{n}|\{k \leq n: \|y_k - \xi \| <
\varepsilon \}| \geq 1 - \delta \\
&\Rightarrow & \frac{1}{n}|\{k \leq n: \|y_k - \xi \| \geq
\varepsilon \}| < 1 -(1 - \delta) = \delta .
\end{eqnarray*}
Thus $\{n \in \mathbb{N}: \frac{1}{n}|\{k \leq n : \|y_k - \xi \|
\geq \varepsilon | \geq \delta \} \subset A $. Since $ A \in I $,
$ \{ n \in \mathbb{N}: \frac{1}{n}|\{ k \leq n: \|y_k - \xi \|
\geq \varepsilon \}| \geq \delta \} \in I $ and so $
{I\mbox{-}st\mbox{-}\lim}~ y = \xi $.
\end{proof}

\begin{defn}
A point $\lambda \in X $ is said to be an $I$-statistical cluster point of a sequencs $x$ in $X$ if for any $\varepsilon > 0 $
\begin{eqnarray*}
d_I(\{k:\|x_k - \lambda \| < \varepsilon \}) \neq 0
\end{eqnarray*}
where $ d_I(A) = I - \underset{n \rightarrow \infty}{\lim}
\frac{1}{n}|\{ k \leq n : k \in A \}$, if exists.
\end{defn}
The set of $I$-statistical cluster point of $x$ is denoted by
${\Lambda}_x^S(I) $.
\begin{thm}
Let $x=\left\{x_k\right\}_{k\in\mathbb N}$ be a sequence in $X$ and $ c \in {\Lambda}_x^S(I) $. Then $ \| \xi - c \| \leq r $ for all $ \xi \in I\mbox{-}st\mbox{-}LIM_x^r $.
\end{thm}

\begin{proof}
If possible, let there exist $ \xi \in I\mbox{-}st\mbox{-}LIM^r_x
$ such that $\| \xi - c \| > r $. Let $ \varepsilon = \frac{\| \xi
- c \| - r}{2} $. Then,
\begin{eqnarray}
\{ k \in \mathbb{N}: \|x_k - \xi \| \geq r + \varepsilon \} \supset \{ k \in \mathbb{N}: \|x_k - c \| < \varepsilon \}.
\end{eqnarray}
Since $ c \in {\Lambda}_x^S(I) $, so $ d_I(\{k:\|x_k - c \| <
\varepsilon \}) \neq 0 $. Hence by (4.8) we have $ d_I(\{k:\|x_k -
\xi \| \geq r + \varepsilon \}) \neq 0 $, which contradicts that $
\xi \in I\mbox{-}st\mbox{-}LIM_x^r $. Hence $ \|\xi - c \| \leq r
$.
\end{proof}

\begin{defn}
A sequence $ x=\left\{x_k\right\}_{k\in\mathbb N} $ in $X$ is said
to be $ I $- statistically bounded if there exists a positive
number $ G $ such that for any $ \delta > 0 $ the set $ A = \{n
\in \mathbb{N}: \frac{1}{n}|\{ k \leq n : \|x_k\| \geq G \}| \geq
\delta\}\in I $.
\end{defn}


\begin{thm}
A sequence $ x=\left\{x_k\right\}_{k\in\mathbb N} $ in $X$ is
$I$-statistically bounded if and only if there exists a non
negative real number $ r > 0 $ such that $
I\mbox{-}st\mbox{-}LIM_x^r \neq \emptyset $.
\end{thm}

\begin{proof}
Let $ x=\left\{x_k\right\}_{k\in\mathbb N} $ be an
$I$-statistically bounded sequence in $X$. Then there exists a
positive real number $G$ such that for $\delta > 0 $ we have $\{ n
: \frac{1}{n}|\{ k \leq n: \|x _k \| \geq G \}| \geq \delta \} \in
I $. Let $ A = \{ k : \|x_k \| \geq G \} $. Then $
I\mbox{-}\underset{ n \rightarrow \infty}{\lim}\frac{1}{n}|\{ k
\leq n: k \in A \}| = 0 $. Let $ r' = sup\{\| x_k\| : k \in
A^c\}$. Then the set $ I\mbox{-}st\mbox{-}LIM^{r'}_x $ contains
the origin. So we have $ I\mbox{-}st\mbox{-}LIM^r_x \neq \emptyset
$ for $ r = r'$.

Conversely, let $ I\mbox{-}st\mbox{-}LIM_x^r \neq \emptyset $ for
some $ r > 0 $. Let $ \xi \in I\mbox{-}st\mbox{-}LIM_x^r $. Choose
$\varepsilon = \| \xi \| $. Then for each $ \delta
> 0 $, $ \{ n \in \mathbb{N}: \frac{1}{n}|\{k \leq n :
\|x_k - \xi \| \geq r + \varepsilon \}| \geq \delta \} \in I $.
Now taking $ G = r + 2 \| \xi \|$, we have $ \{ n \in \mathbb{N}:
\frac{1}{n} | \{ k \leq n : \|x_k\| \geq G \}| \geq \delta \} \in
I $. Therefore $ x $ is $I$-statistically bounded.
\end{proof}

\begin{thm}
Let $ x =\{x_k \}_{k \in \mathbb{N}}$ be a sequence in $X$ and $ r > 0 $.\\
(i) If $ c \in {{\Lambda}_x^S}(I)$, then $ I\mbox{-}st\mbox{-}LIM_x^r \subset \overline{{B_r}(c)} $.\\
(ii) $ I\mbox{-}st\mbox{-}LIM_x^r = \underset{ c \in
{\Lambda}_x^S(I) }{\bigcap}{\overline{{B_r}(c)}} = \{\xi \in X:
{\Lambda}_x^S(I)  \subset \overline{B_r(\xi)}\}$
\end{thm}

\begin{proof}
(i)Let $ c \in {\Lambda}_x^S(I) $. Then by Theorem 4.3.6, for all $
\xi \in I\mbox{-}st\mbox{-}LIM_x^r $, $ \|\xi - c \| \leq r $ and
hence the result follows.

(ii) By (i) it is clear that $ I\mbox{-}st\mbox{-}LIM_r^x \subset
\underset{ c \in {\Lambda}_x^S(I) }{\bigcap}{\overline{{B_r}(c)}}
$. Now for all $ c \in {\Lambda}_x^S(I) $ and $ y \in \underset{ c
\in {\Lambda}_x^S(I) }{\bigcap}{\overline{{B_r}(c)}}$ we have $\|y
- c \| \leq r $. Then clearly $ \underset{ c \in {\Lambda}_x^S(I)
} {\bigcap}{\overline{{B_r}(c)}} \subset \{ \xi \in X:
{\Lambda}_x^S(I) \subset \overline{{B_r}(\xi)}\} $.

Now, let $ y \notin I\mbox{-}st\mbox{-}LIM_x^r $. Then there
exists an $ \varepsilon > 0 $ such that $ {d_I}(A) \neq 0 $, where
$ A = \{ k \in \mathbb{N}: \|x_k - y \| \geq r + \varepsilon \} $
. This implies the existence of an $I$-statistical cluster point
$c$ of the sequence $ x $ with $ \|y - c \| \geq r + \varepsilon
$. This gives ${\Lambda}_x^S(I) \nsubseteq \overline{{B_r}(y)}$and
so $ y \notin \{ \xi \in X: {\Lambda}_x^S(I) \subset
\overline{{B_r}(\xi)} \} $. Hence $ \{ \xi \in X: {\Lambda}_x^S(I)
\subset \overline{{B_r}(\xi)} \} \subset I
\mbox{-}st\mbox{-}LIM_x^r $. This completes the proof.

\end{proof}


\end{document}